\documentclass[11pt]{amsart}

\usepackage{amssymb}
\usepackage{amsfonts}
\usepackage{amsthm}
\usepackage{amsmath}
\usepackage{bbm}
\usepackage{xcolor}
\usepackage{tikz}
\usepackage{hyperref}
\usepackage{ulem}

\usepackage{mathtools} \mathtoolsset{showonlyrefs} 

\usepackage{geometry} 

\usepackage{listings}
\usepackage{hyperref}

\newtheorem{theorem}{Theorem}[section]
\newtheorem{proposition}[theorem]{Proposition}

\newtheorem{lemma}[theorem]{Lemma}

\theoremstyle{definition}
\newtheorem{definition}[theorem]{Definition}

\theoremstyle{remark}
\newtheorem{remark}{Remark}[section]

\numberwithin{equation}{section}

\newcommand{\R}{\mathbb{R}}

\newcommand{\N}{\mathbb{N}}

\renewcommand{\d}{\,{\rm d}}

 \usepackage[shortlabels]{enumitem}
                    \setlist[enumerate, 1]{1\textsuperscript{o}}

\begin{document}

\title[Sharp spherical extension in arbitrary dimensions]{On sharp Fourier extension from spheres in arbitrary dimensions}

\author[Carneiro, Negro and Oliveira e Silva]{Emanuel Carneiro, Giuseppe Negro and Diogo Oliveira e Silva}

\address{
ICTP - The Abdus Salam International Centre for Theoretical Physics,
Strada Costiera, 11, I - 34151, Trieste, Italy.}
\email{carneiro@ictp.it}

\address{ 
Center for Mathematical Analysis, Geometry and Dynamical Systems \& Departamento de Matemática\\ 
Instituto Superior Técnico\\
Av. Rovisco Pais\\ 
1049-001 Lisboa, Portugal.} 
\email{giuseppe.negro@tecnico.ulisboa.pt}
\email{diogo.oliveira.e.silva@tecnico.ulisboa.pt}

\allowdisplaybreaks
\numberwithin{equation}{section}

\begin{abstract}
We prove a new family of sharp $L^2(\mathbb S^{d-1})\to L^4(\R^d)$ Fourier extension inequalities from the unit sphere $\mathbb S^{d-1}\subset \R^d$, valid in arbitrary dimensions $d\geq 3$.
\end{abstract}

\maketitle

\section{Introduction}
The purpose of this short note is to establish the following result.
\begin{theorem}\label{thm_main}
    Let   $d\ge 3$. There exists $a_\star=a_\star(d)>0$ such that 
    \begin{equation}\label{eq_main_conclusion}
        \frac{\displaystyle \int_{\mathbb R^d} \lvert \widehat{f\sigma}(x)\rvert^4\, \frac{\d x}{(2\pi)^d} + a_\star\left\lvert \int_{\mathbb S^{d-1}} f(\omega)\, \d\sigma\right\rvert^4}{\lVert f \rVert_{L^2(\mathbb S^{d-1})}^4 }
        \le 
        \frac{\displaystyle  \int_{\mathbb R^d} \lvert \widehat{\sigma}(x)\rvert^4\, \frac{\d x}{(2\pi)^d} + a_\star\left\lvert   \int_{\mathbb S^{d-1}} \mathbf 1(\omega)\, \d\sigma\right\rvert^4}{\lVert \mathbf 1 \rVert_{L^2(\mathbb S^{d-1})}^4}.
    \end{equation}
    Equality holds in \eqref{eq_main_conclusion} if and only if $f$ is a constant function.
\end{theorem}
In \eqref{eq_main_conclusion}, $\widehat{f\sigma}$ denotes the Fourier extension operator from the unit sphere $\mathbb S^{d-1}\subset\mathbb R^d$, 
\begin{equation}\label{eq_ExtensionDef}
    \widehat{f\sigma}(x)=\int_{\mathbb S^{d-1}} f(\omega)e^{i\omega\cdot x}\,\d\sigma(\omega).
\end{equation}
Theorem \ref{thm_main} is known to hold with $a_\star=0$ when $d=3$, established in \cite{Fo15}, and when $d\in\{4,5,6,7\}$, established in \cite{COS15}.\footnote{A consequence of the modulation symmetry of \eqref{eq_main_conclusion} when $a_\star=0$ is  that the complete set of maximizers then consists of {\it characters}, i.e.\@ functions of the form $f(\omega)=c\exp(i\xi\cdot\omega)$ for some $(c,\xi)\in(\mathbb C\setminus\{0\})\times\R^d$; see \cite{CS12b, OSQ21} in addition to the aforementioned \cite{COS15, Fo15}.}
More recently, Theorem \ref{thm_main} was established when $d=8$ for all $a_\star>\frac{2^{25}\pi^2}{5^27^2 11}$ \cite[Theorem 4]{CNOS21}. 

\medskip

The proof of Theorem \ref{thm_main} relies on the following elementary observation. If, for each $1\leq n\leq N$, 
\begin{equation}\label{eq_inequality_summand}
    \frac{\int_{\mathbb R^d} \lvert \widehat{f\sigma}(x)\rvert^4h_n(x)\, \d x}{\lVert f\rVert_{L^2(\mathbb S^{d-1})}^4} 
    \le 
    \frac{\int_{\mathbb R^d} \lvert \widehat{\sigma}(x)\rvert^4h_n(x)\, \d x}{\lVert \mathbf 1\rVert_{L^2(\mathbb S^{d-1})}^4},
\end{equation}
then it also holds that
\begin{equation}\label{eq_summing_inequalities}
    \frac{\int_{\mathbb R^d} \lvert \widehat{f\sigma}(x)\rvert^4\sum_{n=1}^N h_n(x)\, \d x}{\lVert f\rVert_{L^2(\mathbb S^{d-1})}^4} 
    \le 
    \frac{\int_{\mathbb R^d} \lvert \widehat{\sigma}(x)\rvert^4\sum_{n=1}^N h_n(x)\, \d x}{\lVert \mathbf 1\rVert_{L^2(\mathbb S^{d-1})}^4}.
\end{equation}
Inequality~\eqref{eq_main_conclusion} is of the form \eqref{eq_summing_inequalities}, for a certain family of weights $\{h_n\}_{n=1}^N$, which we will determine, satisfying  
\begin{equation}\label{eq_main_conclusion_delta_version}
    \sum_{n=1}^N h_n=\frac{\bf 1}{(2\pi)^d}+ a{\boldsymbol \delta},
\end{equation}
where $\boldsymbol{\delta}$ denotes the Dirac measure on $\mathbb R^d$.
\begin{remark}\label{rem:motivating_remark}
    It follows from the previous observation that \eqref{eq_main_conclusion} continues to hold if $a_\star$ is replaced by any larger value. Indeed, for each $\lambda\ge 0$ we have
    \begin{equation}\label{eq:explain_larger_astar}
        \frac{\displaystyle \int_{\mathbb R^d} \lvert \widehat{f\sigma}\rvert^4\, \frac{\d x}{(2\pi)^d} + (a_\star+\lambda)\left\lvert \int_{\mathbb S^{d-1}} f\, \d\sigma\right\rvert^4}{\lVert f \rVert_{L^2(\mathbb S^{d-1})}^4 } = 
          \frac{\displaystyle \int_{\mathbb R^d} \lvert \widehat{f\sigma}\rvert^4\, \frac{\d x}{(2\pi)^d} + a_\star\left\lvert \int_{\mathbb S^{d-1}} f\, \d\sigma\right\rvert^4}{\lVert f \rVert_{L^2(\mathbb S^{d-1})}^4 } +
            \lambda\frac{\displaystyle  \left\lvert \int_{\mathbb S^{d-1}} f\, \d\sigma\right\rvert^4}{\lVert f \rVert_{L^2(\mathbb S^{d-1})}^4 },
    \end{equation}
    and Hölder's inequality implies that the latter summand on the right-hand side is  maximized when $f=\boldsymbol{1}$. It follows that the left-hand side is likewise maximal at $f=\boldsymbol{1}$. 
    
    \medskip
    
    It is conjectured that~\eqref{eq_main_conclusion} should hold with $a_\star=0$ in all dimensions $d\geq 3$, that is, constant functions should be global maximizers for the $L^2\to L^4(\R^d)$ extension inequality from $\mathbb S^{d-1}$. Theorem~\ref{thm_main} serves as an intermediate step, and adds weight to this conjecture.
\end{remark}

The analysis in \cite{Fo15} for $d=3$  consists of three main steps: a {\it magical identity}, an ingenious application of the Cauchy--Schwarz inequality, and a careful spectral analysis of the resulting quadratic form.
As observed in \cite{COS15}, the main obstruction for the analysis in \cite{Fo15} to work in dimensions $d\geq 8$ is tied to the fact that the eigenvalues resulting from the third step  do not have the correct signs.  
In \cite{CNOS21}, we 
circumvented this issue when $d=8$ via a careful application of the magical identity combined with a {\it non-magical identity}. In the present paper, we handle the higher dimensional case $d>8$ via an inductive procedure which consists in alternatingly applying magical and non-magical identities to expressions involving the functions $h_n$ from \eqref{eq_main_conclusion_delta_version}. This is coupled with an asymptotic spectral analysis that ensures that the eigenvalues {\it eventually} have the desired sign, leading to inequality \eqref{eq_inequality_summand} and therefore to \eqref{eq_summing_inequalities}, thus concluding the proof of Theorem \ref{thm_main}. We remark that this method can also be applied for $d\in \{3,4,5,6,7\}$. In this case, all eigenvalues have the correct sign;
the above inductive procedure thus reduces to a single step, proving Theorem~\ref{thm_main} with $a_\star=0$. This is an alternative proof of~\cite{Fo15} (for $d=3$) and~\cite{COS15} (for the remaining values of $d$). Our proof is different in two aspects. Firstly, our magical identity is now obtained via integration by parts, importing the method of~\cite{CNOS21}. Secondly, our spectral analysis is also based on integration by parts, via the Rodrigues formula. Note, however, that in the present paper we do not address the complete characterization of maximizers in the $a_\star=0$ case; for that, see~\cite{COS15}. 

    \medskip

The paper is structured as follows. In \S\ref{sec_proof}, we prove Theorem \ref{thm_main} modulo the proofs of the eigenvalue Lemmata \ref{lem_FoschiEvNonPos} and \ref{lem_eigenvalues}. This in turn is accomplished in
\S\ref{sec_evs}.

\section{Proof of Theorem \ref{thm_main}}\label{sec_proof}

We say that $h$ is an {\it admissible weight} if it is of the form
\begin{equation}\label{eq_admissible_weights}
    \widehat{h}=\boldsymbol{\delta}+ \sum_{j=0}^{J}C_{2j}\lvert \cdot\rvert^{2j},
\end{equation}
where $J\in\mathbb N$ and $C_{2j}\in \mathbb R$ are such that $\widehat{h}\ge 0$ on the closed ball $\overline{B}_4:=\{\xi\in\mathbb R^d:\ \lvert\xi\rvert\le 4\}$.
Admissible weights are related to  assumptions (R1)--(R2) from \cite[\S1]{CNOS21}. In this paper we will only consider admissible weights, which in particular enables us to restrict attention to  nonnegative even functions $f\in L^2(\mathbb S^{d-1})$; see \cite[\S2.1--\S2.2]{CNOS21} for details.
In this case, recalling \eqref{eq_ExtensionDef}, we have that $\widehat{f\sigma}$ is then real-valued, and Fubini's theorem implies
    \begin{equation}\label{eq_non_magical_id}
        \tag{NM}
        \int_{\mathbb R^d} \left( \widehat{f\sigma}(x)\right)^4 h(x)\, \d x = \int_{(\mathbb S^{d-1})^4} \widehat{h}\left(\sum_{j=1}^4 \omega_j\right) \prod_{j=1}^4 f(\omega_j)\d\sigma(\boldsymbol{\omega}),
    \end{equation}
where we abbreviated $\d\sigma(\boldsymbol{\omega}):=\prod_{j=1}^4 \d\sigma(\omega_j)$.
We refer to \eqref{eq_non_magical_id} as the {\it non-magical identity}. 
Its {\it magical identity} counterpart was  observed in \cite[Eq.~(3.10)]{CNOS21}:
    \begin{equation}\label{eq_magical_id}
        \tag{M}
        \int_{\mathbb R^d} \left( \widehat{f\sigma}(x)\right)^4 h(x)\, \d x = \int_{(\mathbb S^{d-1})^4} \widehat{h}\left(\sum_{j=1}^4 \omega_j\right) M(\boldsymbol{\omega})\prod_{j=1}^4 f(\omega_j)\d\sigma(\boldsymbol{\omega}),
    \end{equation}
    where the function $M$ is given by
\begin{equation}\label{eq_bigM}
    M(\boldsymbol{\omega}):=\frac14\left( \lvert \omega_1+\omega_2\rvert^2+\lvert \omega_3+\omega_4\rvert^2 -(\omega_1+\omega_2)\cdot(\omega_3+\omega_4)\right).
\end{equation}
We write $\d\sigma(\boldsymbol{\omega}_{ij}):=\d\sigma(\omega_i)\d\sigma(\omega_j)$ and, in connection to \eqref{eq_magical_id}--\eqref{eq_non_magical_id}, consider  the kernels 
\begin{align}
    &K_{\mathrm{M}}(\omega_1\cdot\omega_2):=\int_{(\mathbb S^{d-1})^2} \widehat{h}\left(\sum_{j=1}^4 \omega_j\right) M(\boldsymbol{\omega})\d\sigma(\boldsymbol{\omega}_{34});\label{eq_mag_kernel}\\
        &K_{\mathrm{NM}}(\omega_1\cdot\omega_2):=\int_{(\mathbb S^{d-1})^2} \widehat{h}\left(\sum_{j=1}^4 \omega_j\right) \d\sigma(\boldsymbol{\omega}_{34}).\label{eq:non_mag_kernel}
\end{align}
Let $\Lambda_{\mathrm{M}}(k), \Lambda_{\mathrm{NM}}(k)$ for  $k\in \mathbb N$ denote the corresponding eigenvalues; see \eqref{eq_lambda_k_definition} below for the precise definition. The fact that $f$ is an even function implies that only those eigenvalues corresponding to even values of $k$ will be of importance.
The following result motivates the rest of the analysis.
\begin{proposition}\label{prop_fund_mechanism}
 If $\Lambda_{\mathrm{M}}(2\ell)\le 0$ for all $\ell\ge 1$, or $\Lambda_{\mathrm{NM}}(2\ell)\le 0$ for all $\ell\ge 1$,  then for every $f\in L^2(\mathbb S^{d-1})$ we have that
\begin{equation}\label{eq:constants_maximize}
    \frac{\displaystyle\int_{\mathbb R^d} \lvert \widehat{f\sigma}(x)\rvert^4h(x)\, \d x}{\lVert f\rVert_{L^2(\mathbb S^{d-1})}^4} 
    \le 
    \frac{\displaystyle\int_{\mathbb R^d} \lvert \widehat{\sigma}(x)\rvert^4h(x)\, \d x}{\lVert \mathbf 1\rVert_{L^2(\mathbb S^{d-1})}^4}.
\end{equation}
\end{proposition}
\noindent The proof of Proposition \ref{prop_fund_mechanism} follows the general strategy outlined in \cite{Fo15} and can be extracted from \cite{CNOS21}; as such, we shall be brief.
\begin{proof}[Proof sketch of Proposition \ref{prop_fund_mechanism}]
 We continue to assume that $f$ is  nonnegative and even. Additionally, we  require $f\in L^4(\mathbb S^{d-1})$; this can  be  removed via a standard density argument as in \cite[Proof of Lemma 12]{COS15}.
Consider the case when $\Lambda_{\mathrm{M}}(2\ell)\le 0$, for all $\ell\ge 1$. Applying the Cauchy--Schwarz inequality to the quartic form~\eqref{eq_magical_id} yields
    \begin{equation}\label{eq_magical_cauchy_schwarz}
    \int_{(\mathbb S^{d-1})^4} \widehat{h}\left(\sum_{j=1}^4 \omega_j\right) M(\boldsymbol{\omega})\prod_{j=1}^4 f(\omega_j)\d\sigma(\boldsymbol{\omega}) \le \int_{(\mathbb S^{d-1})^2} f^2(\omega_1)f^2(\omega_2)K_{\mathrm{M}}(\omega_1\cdot \omega_2) \d\sigma(\boldsymbol{\omega}_{12}),
\end{equation}
where the kernel $K_{\mathrm{M}}$ was defined in \eqref{eq_mag_kernel}.
Since $f^2$ is even and square-integrable, we can decompose it in spherical harmonics as $f^2=\sum_{\ell=0}^\infty Y_{2\ell}$. By the Funk--Hecke formula, see \eqref{eq_FH} below, and the assumption on the eigenvalues of $K_{\mathrm{M}}$, we have that
\begin{equation}\label{eq_magical_conclusion}
    \begin{split}
        \int_{(\mathbb S^{d-1})^2} f^2(\omega_1)f^2(\omega_2)K_{\mathrm{M}}(\omega_1\cdot \omega_2) \d\sigma(\boldsymbol{\omega}_{12})&=\Lambda_{\mathrm{M}}(0)\lVert Y_0\rVert_{L^2(\mathbb S^{d-1})}^2 +\sum_{\ell = 1}^\infty\Lambda_{\mathrm{M}}(2\ell) \lVert Y_{2\ell}\rVert_{L^2(\mathbb S^{d-1})}^2\\
        &\le \Lambda_{\mathrm{M}}(0)\lVert Y_0\rVert_{L^2(\mathbb S^{d-1})}^2,
   \end{split}
\end{equation}
from which~\eqref{eq:constants_maximize} follows. 
The case when  $\Lambda_{\mathrm{NM}}(2\ell)\le 0$ for all  $\ell\ge 1$ is analogous, one just has to invoke \eqref{eq_non_magical_id} instead of \eqref{eq_magical_id} in \eqref{eq_magical_cauchy_schwarz}. This concludes the sketch of the proof of Proposition \ref{prop_fund_mechanism}.
\end{proof}
The constant weight $h=\boldsymbol{1}/(2\pi)^d$ is especially important. Since $f$ is nonnnegative and even, we apply~\eqref{eq_magical_id} and the Cauchy--Schwarz inequality as in the  proof of Proposition \ref{prop_fund_mechanism} to obtain
\begin{equation}\label{eq:reproving_Foschi}
    \begin{split}
        \int_{\mathbb R^d} \left( \widehat{f\sigma}(x)\right)^4\, \frac{\d x}{(2\pi)^d}&=\int_{(\mathbb S^{d-1})^4} \boldsymbol\delta\left(\sum_{j=1}^4 \omega_j\right) M(\boldsymbol{\omega})\prod_{j=1}^4 f(\omega_j)\d\sigma(\boldsymbol{\omega}) \\
        &\le \int_{(\mathbb S^{d-1})^2} K_{\boldsymbol{1}}(\omega_1\cdot\omega_2)f^2(\omega_1)f^2(\omega_2)\, \d\sigma(\boldsymbol{\omega}_{12}),
    \end{split}
\end{equation}
where the kernel $K_{\boldsymbol{1}}$ is given by
\begin{equation}\label{eq:free_kernel}
    K_{\boldsymbol{1}}(\omega_1\cdot\omega_2):=\int_{(\mathbb S^{d-1})^2} \boldsymbol\delta\left(\sum_{j=1}^4 \omega_j\right)M(\boldsymbol{\omega}) \d\sigma(\boldsymbol{\omega}_{34}).
\end{equation}
Denote the even eigenvalues of $K_{\boldsymbol{1}}$ by  $\{\lambda_{\boldsymbol{1}}(2\ell)\}_{\ell\geq 1}$. These are not necessarily all nonpositive if the dimension $d$ is sufficiently large,\footnote{If this were the case, Proposition \ref{prop_fund_mechanism} would allows us to  obtain Theorem~\ref{thm_main}  for every $a_\star\geq 0$.} but it turns out that they are {\it eventually nonpositive}. This is the content of the next result, whose proof is deferred to the forthcoming \S\ref{sec_evs}.

\begin{lemma}\label{lem_FoschiEvNonPos}
For every $d\geq 3$, there exists $\ell_\star(d)\in\N$, such that $\lambda_{\boldsymbol{1}}(2\ell)\le 0$ for all $\ell> \ell_\star(d)$. 
Moreover, $\ell_\star(d)=O(d)$. For $d\in \{3,4,5,6,7\}$, one can take $\ell_\star(d)=0$.
\end{lemma}

We also need to control the sign of the (even) eigenvalues of the  kernels
\begin{equation}\label{eq_kernels}
    \begin{split}
        K_n(\omega_1\cdot\omega_2):=\int_{(\mathbb S^{d-1})^2}\left\lvert \sum_{j=1}^4 \omega_j \right\rvert^n M(\boldsymbol{\omega})\, \d\sigma(\boldsymbol{\omega}_{34}), \\ 
        L_n(\omega_1\cdot\omega_2):=\int_{(\mathbb S^{d-1})^2}\left\lvert \sum_{j=1}^4 \omega_j \right\rvert^n \, \d\sigma(\boldsymbol{\omega}_{34}),
    \end{split}
\end{equation}
respectively denoted by $\lambda_n(2\ell)$ and $\mu_n(2\ell)$. It suffices to consider the case of even $n$, which is the content of the next result, proved in  \S\ref{sec_evs}.
\begin{lemma}\label{lem_eigenvalues} Let $k\in2\mathbb{Z}_{\geq 0}$ be even and let $m\in\mathbb{Z}_{\geq 0}$.
    The eigenvalues of $K_{2m}$ satisfy  
\begin{equation}\label{eq:magic_eigenvalue_lemma}
        \begin{split}
            \lambda_{2m}(k) >0,& \quad k=m+1,\\ 
            \lambda_{2m}(k)=0, & \quad k >m+1,
        \end{split}
    \end{equation}
    and the eigenvalues of $L_{2m}$ satisfy
\begin{equation}\label{eq:nonmagic_eigenvalue_lemma}
        \begin{split}
            \mu_{2m}(k) >0,& \quad k=m,\\ 
             \mu_{2m}(k)=0, & \quad k >m.
        \end{split}
    \end{equation}
\end{lemma}




We are now ready to prove inequality~\eqref{eq_main_conclusion}. From Lemma \ref{lem_FoschiEvNonPos}  we know that, for $d\in \{3,4,5,6,7\}$, $\lambda_{\boldsymbol{1}}(2\ell)\le 0$ for all $\ell >0$. In this case, Proposition~\ref{prop_fund_mechanism} directly applies to the weight $h=1$, concluding the proof of Theorem~\ref{thm_main} with $a_\star=0$.

For $d\ge 8$, Lemma~\ref{lem_FoschiEvNonPos} guarantees the existence of    $N\in\mathbb N$ such that $\lambda_{\boldsymbol{1}}(2\ell)\le 0$ for $2\ell >N$. The proof of Lemma \ref{lem_FoschiEvNonPos} reveals that a valid choice is $N=\frac{d-3}2$ if $d$ is odd and $N=\frac{d-4}2$ if $d$ is even. Since $d\geq 8$, we may therefore assume that $N\geq 1$. With a view towards applying Proposition \ref{prop_fund_mechanism}, define functions $h_n$ for $1\leq n\leq 2N$ via their Fourier transforms as
\begin{equation}\scalebox{0.9}  
    {$
        \begin{array}{rlcccccc}
            \widehat{h}_1(\xi)&=\boldsymbol{\delta} &-C_{1, 4N-2}\lvert\xi\rvert^{4N-2} &-C_{1, 4N-4}\lvert\xi\rvert^{4N-4}&-C_{1, 4N-6}\lvert\xi\rvert^{4N-6}&-\ldots&-C_{1, 2}\lvert\xi\rvert^{2}&+C_{1,0}\\ 
            \widehat{h}_2(\xi)&=       &+C_{2, 4N-2}\lvert\xi\rvert^{4N-2} &-C_{2, 4N-4}\lvert\xi\rvert^{4N-4}&-C_{2, 4N-6}\lvert\xi\rvert^{4N-6}&-\ldots&-C_{2, 2}\lvert\xi\rvert^{2}&+C_{2,0}\\
            \widehat{h}_3(\xi)&=    &   &+C_{3, 4N-4}\lvert\xi\rvert^{4N-4}&-C_{3, 4N-6}\lvert\xi\rvert^{4N-6}&-\ldots&-C_{3, 2}\lvert\xi\rvert^{2}&+C_{3,0}\\
            \widehat{h}_4(\xi)&=    &   &&+C_{4, 4N-6}\lvert\xi\rvert^{4N-6}&-\ldots&-C_{4, 2}\lvert\xi\rvert^{2}&+C_{4,0}\\
            &\,\vdots& \\
            \widehat{h}_{2N}(\xi) &=&&&&&+C_{2N, 2}\lvert\xi\rvert^2&+C_{2N, 0},
        \end{array}
   $}
\end{equation}
and respectively let $\Lambda_{n, \mathrm{M}}$ and $\Lambda_{n, \mathrm{NM}}$  denote the eigenvalues of the kernels~\eqref{eq_mag_kernel} and~\eqref{eq:non_mag_kernel} corresponding to the function $h_n$. The nonnegative constants $C_{n,m}\ge 0$ will be chosen below, recalling that our target weight is $h=\frac{\boldsymbol{1}}{(2\pi)^d} +a\boldsymbol{\delta}$. Our first requirement is  \[\widehat{h}=\boldsymbol{\delta} +a\boldsymbol{1}=\sum_{n=1}^{2N}\widehat{h}_n,\] which translates into
\begin{equation}\label{eq:sum_condition}
    \begin{split}
        C_{2, 4N-2}&= C_{1, 4N-2}, \\ 
        C_{3, 4N-4}&=C_{1, 4N-4}+C_{2, 4N-4},\\ 
        &\vdots \\ 
        C_{2N, 2}&=C_{1, 2}+C_{2, 2}+\ldots+C_{2N-1, 2}.
    \end{split}
\end{equation}
In particular, Theorem~\ref{thm_main} will be proved for $a_\star=\sum_{n=1}^{2N}C_{n, 0}$.
The coefficients $(C_{n,m})$ will be chosen in such a way that the following two conditions hold for every $1\leq n\leq 2N$:
\medskip
\begin{enumerate}
        \item[(Adm)] $\widehat{h}_n\ge 0$ on $\overline{B}_4$;
        \item[(Eig)] For every  $\ell\ge 1$, $\displaystyle
            \begin{cases} 
                \Lambda_{n, \mathrm{M}}(2\ell)\le 0, & \text{if }n \text{ is odd,} \\ 
                \Lambda_{n, \mathrm{NM}}(2\ell)\le 0, & \text{if }n \text{ is even.}
            \end{cases}
            $
\end{enumerate}
\medskip
Condition (Adm) ensures that $h_n$ is an admissible weight; recall \eqref{eq_admissible_weights}.
Once this construction is done, we can apply Proposition~\ref{prop_fund_mechanism} to each of the $h_n$, thus completing the proof of Theorem~\ref{thm_main}. 

\medskip

In the following inductive scheme we adopt the following notational conventions:
\begin{itemize}
    \item Constants $C_{n,m}$ which remain undefined  are set to  $0$.
    \item  $(C_p\lambda)_q(r):=C_{p, q}\lambda_q(r)$ and $(C_p\mu)_q(r):=C_{p, q}\mu_q(r)$.
    \item  $\{x\}_+:=\max\{x, 0\}$.
\end{itemize}

\medskip
    
We start by constructing $(h_1, h_2)$, then we will turn to the remaining $(h_{2n-1}, h_{2n})$ for $2\leq n< N$, and finally to the last pair $(h_{2N-1}, h_{2N})$. In order to ensure that $h_1$ satisfies (Eig), we start by noting that, for $\ell\ge 1$, 
\begin{equation}\label{eq_h_one_eigenvals}
\Lambda_{1, \mathrm{M}}(2\ell)=\lambda_{\boldsymbol{1}}(2\ell) - (C_{1}\lambda)_{4N-2}(2\ell) -(C_{1}\lambda)_{4N-4}(2\ell)-\ldots-(C_{1}\lambda)_2(2\ell).
\end{equation}
By~\eqref{eq:magic_eigenvalue_lemma},  every term on the right-hand side of \eqref{eq_h_one_eigenvals} vanishes for $\ell>N$, except possibly for $\lambda_{\boldsymbol{1}}(2\ell)$, which is nonpositive. Thus $\Lambda_{1, \mathrm{M}}(2\ell)\le 0$ for $\ell>N$. If $\ell=N$, then
\begin{equation}\label{eq:h_one_dominant_eigenval}
    \Lambda_{1, \mathrm{M}}(2N)=\lambda_{\boldsymbol{1}}(2N)-(C_{1}\lambda)_{4N-2}(2N),
\end{equation}
and in light of ~\eqref{eq:magic_eigenvalue_lemma} we have that $\lambda_{4N-2}(2N)>0$. We therefore let 
\begin{equation}\label{eq_choose_first_coeff}
    C_{1, 4N-2}:= \left\{ \frac{\lambda_{\boldsymbol{1}}(2N)}{\lambda_{4N-2}(2N)}\right\}_+, 
\end{equation}
thus ensuring $\Lambda_{1, \mathrm{M}}(2N)\le 0$. Again by~\eqref{eq:magic_eigenvalue_lemma}, for $0<m< N$ we have $\lambda_{4N-2-4m}(2(N-m))>0$ and so we can let
\begin{equation}\label{eq_choose_second_coeff}
    C_{1, 4N-2-4m}:= \left\{\frac{\lambda_{\boldsymbol{1}}(2(N-m)) - \sum_{k=1}^{2m} (C_1\lambda)_{4N-2-4m+2k}(2(N-m))}{\lambda_{4N-2-4m}(2(N-m))}\right\}_+,
\end{equation}
ensuring $\Lambda_{1, \mathrm{M}}(2(N-m))\le 0$. Thus $h_1$ satisfies (Eig). Finally, we choose $C_{1, 0}\ge 0$ sufficiently large to ensure that $h_1$ satisfies (Adm) as well.

\medskip

We now turn to $h_2$. Here we apply the non-magical identity~\eqref{eq_non_magical_id}. For $\ell\ge 1$, we have
\begin{equation}\label{eq:h_two_eigenvals}
    \Lambda_{2, \mathrm{NM}}(2\ell)=(C_2\mu)_{4N-2}(2\ell) - (C_2\mu)_{4N-4}(2\ell) - \ldots - (C_2\mu)_{ 2}(2\ell).
\end{equation}
By~\eqref{eq:nonmagic_eigenvalue_lemma}, $\Lambda_{2, \mathrm{NM}}(2\ell)=0$ for $\ell> N-1$. Next we need to ensure that $\Lambda_{2, \mathrm{NM}}(2(N-1))\le 0$. Note that the value of $C_{2, 4N-2}$ is dictated by~\eqref{eq:sum_condition}; we cannot make any choice here. However, $\mu_{4N-4}(2(N-1))>0$, and so we let  
\begin{equation}\label{eq:choose_first_coeff_two}
    C_{2, 4N-4}:=\left\{\frac{(C_2\mu)_{ 4N-2}(2(N-1))}{\mu_{4N-4}(2(N-1))}\right\}_+,
\end{equation}
ensuring the desired $\Lambda_{2, \mathrm{NM}}(2(N-1))\le 0$. For  $0<m<N-1$, we then let
\begin{equation}\label{eq:choose_second_coeff_two}
    \scalebox{1}
    {$
        \begin{split}
            &C_{2, 4N-4-4m}:=\\
            &\left\{\frac{ (C_2\mu)_{4N-2}(2(N-1-m))-\sum_{k=1}^{2m} (C_2\mu)_{4N-4-4m+2k}(2(N-1-m))}{\mu_{4N-4-4m}(2(N-1-m))}\right\}_+,
        \end{split}
    $}
\end{equation}
thus ensuring that $\Lambda_{2, \mathrm{NM}}(2(N-1-m))\le 0$. So $h_2$ satisfies (Eig). Finally we choose $C_{2, 0}\ge 0$ sufficiently large so that $h_2$ satisfies (Adm) as well.

\medskip

We turn to constructing the pair $(h_{2n-1}, h_{2n})$ for each $2\leq n< N$:
\begin{equation}\scalebox{0.9}
        {$
        \begin{array}{rllllll}
            \widehat{h}_{2n-1}(\xi)&=C_{2n-1, 4N-4n+4}\lvert\xi\rvert^{4N-4n+4} &-C_{2n-1, 4N-4n+2}\lvert\xi\rvert^{4N-4n+2}&-\ldots&-C_{2n-1, 2}\lvert\xi\rvert^{2}&+C_{2n-1,0},\\ 
            \widehat{h}_{2n}(\xi)&=       &+C_{2n, 4N-4n+2}\lvert\xi\rvert^{4N-4n+2} &-\ldots&-C_{2n, 2}\lvert\xi\rvert^{2}&+C_{2n,0}.\\
        \end{array}
        $}
\end{equation}
At each step, the constant associated to  the highest degree term will be dictated by~\eqref{eq:sum_condition}, and we will choose the subsequent ones to ensure that (Eig) holds. Finally, we will choose the last constant $C_{\ast, 0}$ $(\ast\in\{2n-1,2n\})$ so that (Adm) holds. 
For $\ell\ge 1$, 
\begin{equation}\label{eq:general_odd_eigenvals}
    \Lambda_{2n-1, \mathrm{M}}(2\ell)=(C_{2n-1}\lambda)_{4N-4n+4}(2\ell) - (C_{2n-1}\lambda)_{4N-4n+2}(2\ell)-\ldots-(C_{2n-1}\lambda)_2(2\ell),
\end{equation}
so by~\eqref{eq:magic_eigenvalue_lemma} $\Lambda_{2n-1, \mathrm{M}}(2\ell)=0$ for $\ell> N-n+1$. The top coefficient $C_{2n-1, 4N-4n+4}$ is dictated by~\eqref{eq:sum_condition}. Since $\lambda_{4N-4n+2}(2(N-n+1))>0$, we  let
\begin{equation}\label{eq:choose_main_odd_coeff}
    C_{2n-1, 4N-4n+2}:=\left\{ \frac{(C_{2n-1}\lambda)_{4N-4n+4}(2(N-n+1))}{\lambda_{4N-4n+2}(2(N-n+1))}\right\}_+, 
\end{equation}
ensuring $\Lambda_{2n-1, \mathrm{M}}(2(N-n+1))\le 0$. By the same logic, for $0<m<N-n+1$ we let
\begin{equation}\label{eq_choose_odd_coeff} 
    \scalebox{0.98}{$
    \begin{split}
            &C_{2n-1, 4N-4n+2-4m}:=
            \\
            &\left\{ \frac{(C_{2n-1}\lambda)_{4N-4n+4}(2(N-n+1-m))-\sum_{k=1}^{2m}(C_{2n-1}\lambda)_{4N-4n+2-4m+2k}(2(N-n+1-m))}{\lambda_{4N-4n+2-4m}(2(N-n+1-m))}\right\}_+
        \end{split}
        $}
\end{equation}
ensuring $\Lambda_{2n-1, \mathrm{M}}(2(N-n+1-m))\le 0$. Thus $h_{2n-1}$ satisfies (Eig). Finally we choose $C_{2n-1,0}>0$ sufficiently large to ensure (Adm) as well. 

\medskip

To construct $h_{2n}$, we note that, for $\ell\ge 1$, 
\begin{equation}\label{eq_general_even_eigenvals}
    \Lambda_{2n, \mathrm{NM}}(2\ell)=(C_{2n}\mu)_{4N-4n+2}(2\ell) - (C_{2n}\mu)_{4N-4n}(2\ell)-\ldots-(C_{2n}\mu)_2(2\ell),
\end{equation}
so by~\eqref{eq:nonmagic_eigenvalue_lemma} $\Lambda_{2n, \mathrm{NM}}(2\ell)=0$ for $\ell> N-n$. The coefficient $C_{2n, 4N-4n+2}$ is dictated by~\eqref{eq:sum_condition}; since $\mu_{4N-4n}(2(N-n))>0$, letting 
 \begin{equation}\label{eq_choose_main_even_coeff}
    C_{2n, 4N-4n}:=\left\{ \frac{(C_{2n}\mu)_{4N-4n+2}(2(N-n))}{\mu_{4N-4n}(2(N-n))}\right\}_+
 \end{equation}
ensures that $\Lambda_{2n, \mathrm{NM}}(2(N-n))\le 0$. Finally, following the same logic again, we let for $0<m<N-n$, 
\begin{equation}\label{eq_choose_even_coeff}
    \begin{split}
        &C_{2n, 4N-4n-4m}:=
        \\&\left\{ \frac{(C_{2n}\mu)_{4N-4n+2}(2(N-n-m))-\sum_{k=1}^{2m}(C_{2n}\mu)_{4N-4n-4m+2k}(2(N-n-m))}{\mu_{4N-4n-4m}(2(N-n-m))}\right\}_+
    \end{split}
\end{equation}
ensuring $\Lambda_{2n, \mathrm{NM}}(2(N-n-m))\le 0$. Thus $h_{2n}$ satisfies (Eig). Finally we choose $C_{2n,0}>0$ sufficiently large to ensure (Adm). 

\medskip

We end with the last two weights:
\begin{equation}\scalebox{1}
        {$
        \begin{array}{rlll}
            \widehat{h}_{2N-1}(\xi)&=C_{2N-1, 4}\lvert\xi\rvert^{4} &-C_{2N-1, 2}\lvert\xi\rvert^{2}&+C_{2N-1,0},\\ 
            \widehat{h}_{2N}(\xi)&=&+C_{2N, 2}\lvert\xi\rvert^{2}&+C_{2N,0}.\\
        \end{array}
        $}
\end{equation}
For $\ell\ge 1$ we have
\begin{equation}\label{eq:last_magic_eigenval}
    \Lambda_{2N-1, \mathrm{M}}(2\ell)=(C_{2N-1}\lambda)_4(2\ell) - (C_{2N-1}\lambda)_2(2\ell).
\end{equation}
By~\eqref{eq:magic_eigenvalue_lemma}, $\Lambda_{2N-1, \mathrm{M}}(2\ell)=0$ for $\ell>1$. The constant $C_{2N-1, 4}$ is dictated by~\eqref{eq:sum_condition}. We let
\begin{equation}\label{eq:choose_last_constant}
    C_{2N-1, 2}:=\left\{ \frac{(C_{2N-1}\lambda)_4(2)}{\lambda_2(2)}\right\}_+
\end{equation}
ensuring $\Lambda_{2N-1,\mathrm{M}}(2)\le 0$ and so $h_{2N-1}$ satisfies (Eig). Then we let $C_{2N-1,0}>0$ be sufficiently large to satisfy (Adm). Finally, we turn to (for $\ell \geq 1$)
\begin{equation}\label{eq:last_nonmagic_eigenvals}
    \Lambda_{2N, \mathrm{NM}}(2\ell)=(C_{2N}\mu)_2(2\ell), 
\end{equation}
noting that, by~\eqref{eq:nonmagic_eigenvalue_lemma}, it vanishes for all $\ell \geq 1$. So $h_{2N}$ automatically satisfies (Eig). It suffices to choose $C_{2N, 0}>0$ sufficiently large to ensure (Adm). 

\medskip

This completes the proof of Theorem \ref{thm_main} modulo the proof of Lemmata \ref{lem_FoschiEvNonPos} and \ref{lem_eigenvalues}, which is the subject of \S\ref{sec_evs}, and the characterization of maximizers, which  follows from the analysis in \cite[\S8]{CNOS21}.

\section{Proofs of Lemmata \ref{lem_FoschiEvNonPos} and \ref{lem_eigenvalues}}\label{sec_evs}

The Funk--Hecke formula \cite[Theorem 1.2.9]{DX13} has already made an appearance. We state it here in the exact form which is needed for our purposes. Given $\nu\ge 0$, let $\{C_k^{(\nu)}\}_{k\in\mathbb N}$ denote the family of Gegenbauer polynomials, i.e., the system of orthogonal polynomials on $[-1, 1]$ with respect to the measure $(1-t^2)^{\nu-\frac12}\,\d t$. This system is unique up to normalization factors.
  \begin{lemma}[\cite{DX13}]\label{lem_Funk_Hecke}
    Let  $g=\sum_{k=0}^\infty Y_k\in L^2(\mathbb S^{d-1})$, where  $Y_k$ is a spherical harmonic of degree $k$. Let  $K:[-1, 1]\to\R$ be integrable with respect to the measure $(1-t^2)^\frac{d-3}{2}\,\d t$. Then 
    \begin{equation}\label{eq_FH}
        \int_{(\mathbb S^{d-1})^2} g(\omega_1)g(\omega_2) K(\omega_1\cdot \omega_2)\, \d\sigma(\boldsymbol{\omega}_{12}) =\sum_{k=0}^\infty \lambda(k) \lVert Y_k\rVert_{L^2(\mathbb S^{d-1})}^2, 
    \end{equation}
    where the eigenvalues $\lambda(k)$ are given by
    \begin{equation}\label{eq_lambda_k_definition}
        \lambda(k)=\frac{\lvert \mathbb S^{d-2}\rvert}{C_k^{(\frac{d}2 -1)}(1)}\int_{-1}^1 K(t)C_k^{(\frac{d}2 -1)}(t)(1-t^2)^{{\frac{d-3}2}}\, \d t.
    \end{equation}
\end{lemma}

The following notation will be convenient in the course of the proof of Lemma \ref{lem_FoschiEvNonPos}.

\begin{definition}[Falling factorial]\label{def:pokemon}
    Given $a\in \mathbb R$ and $n\in\mathbb N$, let $a^{\underline{n}}:=a(a-1)\ldots (a-n+1)$. 
\end{definition}
\begin{proof}[Proof of Lemma \ref{lem_FoschiEvNonPos}]
    The kernel $K_{\boldsymbol{1}}$ has been computed explicitly in~\cite[Eq.~(4.2)]{CNOS21}: 
\begin{equation}\label{eq_explicit_K_one}
    K_{\boldsymbol{1}}(t)=C_d(1+t)^\frac12(1-t)^\frac{d-3}{2},\,\,\,t\in[-1,1],
\end{equation}
where $C_d>0$ is a constant. In light of \eqref{eq_lambda_k_definition}, we  consider the integral 
\begin{equation}\label{eq:Akd_integral}
    A_k^{(d)}:=\int_{-1}^1 (1-t)^{\frac12}(1+t)^{\frac{d-3}{2}}C_k^{(\frac d2 -1)}(t)(1-t^2)^{\frac{d-3}{2}}\, \d t.
\end{equation}
For even $k$, the change of variables $t\mapsto -t$ reveals that the eigenvalue $\lambda_{\boldsymbol{1}}(k)$ equals a positive multiple of $A_k^{(d)}$. 

\medskip

We claim that there exists $k_\star(d)=O(d)$ such that $A_k^{(d)}\le 0$, for every even integer $k\ge k_\star(d)$. 
To prove this, recall the Rodrigues formula for Gegenbauer polynomials \cite[p.~22]{Mu98},
\begin{equation}\label{eq_Rodrigues}
    C^{(\frac d2 -1)}_k(t)=\frac{(-1)^k R_k^{(\frac d 2 -1)}}{(1-t^2)^\frac{d-3}{2}}\frac{\d^k}{\d t^k}\left( (1-t^2)^{k+\frac{d-3}{2}}\right),
\end{equation}
where $R_k^{(\frac d 2 -1)}>0$ is a positive constant. Inserting this formula into~\eqref{eq:Akd_integral} and integrating by parts, we have that
\begin{equation}\label{eq:Akd_intbyparts}
    A_k^{(d)}=R_k^{(\frac d 2-1)} \int_{-1}^1 \frac{\d^k}{\d t^k}\left((1-t)^\frac12 (1+t)^{\frac{d-3}{2}} \right) (1-t^2)^{k+\frac{d-3}{2}}\, \d t.
\end{equation}
The case of odd dimensions turns out to be simpler, and we handle it first.

\subsection{The odd $d$ case}
If $d=3$ and $k\ge 1$, then
\begin{equation}\label{eq_d_three_base_case}
    \frac{\d^k}{\d t^k}\left( (1-t)^\frac12 \right) = (-1)^k\left(\frac12\right)^{\underline{k}} (1-t)^{\frac12 - k}  <0,
\end{equation}
for every $t\in(-1, 1)$. This immediately implies that $A^{(3)}_k<0$, for all $k\ge 1$.
Now let $d=2n+3$ for $n\ge 1$. The first factor of the integrand in~\eqref{eq:Akd_intbyparts} equals
\begin{equation}\label{eq:Leibniz_rule}
    \frac{\d^k}{\d t^k} \left( (1-t)^\frac12(1+t)^n \right) =\sum_{j=0}^{k-1} \binom{k}{j} \frac{\d^{k-j}}{\d t^{k-j}}\left( (1-t)^\frac12\right) \frac{\d^j}{\d t^j}\left( (1+t)^n\right)+ \frac{\d^k}{\d t^k}\left( (1+t)^n\right).
\end{equation}
In view of \eqref{eq_d_three_base_case}, the only positive summand is the last one, which however vanishes whenever $k>n$. We conclude that the claim holds for every even $k$ satisfying 
$k> n=\frac{d-3}{2}$. In particular, note that $A^{(d)}_{k}\le 0$ for all even $k\ge 2$ provided $d\in \{3, 5\}$. This is also true for $d=7$. Indeed, the previous argument implies $A_k^{(7)}\le 0$ for $k\ge 4$, and for $k=2$ directly computing in~\eqref{eq:Akd_intbyparts} yields
\begin{equation*}
    \begin{split}
        A_2^{(7)} &= R_2^{(\frac 5 2)} \int_{-1}^1 \frac{\d^2}{\d t^2}\left((1-t)^\frac12 (1+t)^{2} \right) (1-t^2)^{4}\, \d t\\
        &=\frac{R_2^{(\frac 5 2)}}{4}\int_{-1}^1(1-t)^{\frac52}(1+t)^4(-1-18t+15t^2)\,\d t \\
        &=-\frac{R_2^{(\frac 5 2)}}{4}\frac{2097152\sqrt{2}}{1322685}<0.
    \end{split}
\end{equation*}
\subsection{The even $d$ case}
Given $n\ge 0$,  let $d=3+(2n+1)$. Applying the binomial theorem (see Remark~\ref{rem:Raabe} below), 
the first factor of the integrand in~\eqref{eq:Akd_intbyparts} equals 
\begin{equation}\label{eq_even_case}
    \frac{\d^k}{\d t^k} \left( (1-t^2)^\frac12(1+t)^n \right)= \sum_{\ell=0}^n \sum_{m=0}^\infty \binom{n}{\ell}\binom{\frac12}{m}(-1)^m\frac{\d^k}{\d t^k}(t^{2m+\ell}).
\end{equation}
Assume that $k>n$ and that $k$ is even. In particular, $\ell<k$, and so all summands  with $m=0$ on the right-hand side of \eqref{eq_even_case} vanish. This is important, because for all $m\ge 1$, 
\begin{equation}\label{eq:crucial_binomial_sign}
    \binom{\frac12}{m}(-1)^m<0.
\end{equation}
Therefore \eqref{eq_even_case} yields 
\begin{equation}\label{eq:even_case_reduced}
    \frac{\d^k}{\d t^k} \left( (1-t^2)^\frac12(1+t)^n \right) = -\sum_{\ell=0}^n\sum_{m=1}^\infty C_{k,\ell,m,n}\mathbf 1_{\{k\le 2m+\ell\}}t^{2m+\ell-k},
\end{equation}
for some constants $C_{k,\ell,m,n}> 0$. Since $k$ is even, we then have that 
\begin{equation}\label{eq:end_even_case}
    A_k^{(d)}= -R_k^{(\frac{d}{2}-1)}\sum_{\ell=0}^n\sum_{m=1}^\infty C_{k,\ell,m,n} \mathbf 1_{\{k\le 2m+\ell\}}\int_{-1}^1 t^{2m+\ell-k}(1-t^2)^{k+\frac{d-3}{2}}\, \d t <0,
\end{equation}
because all summands corresponding to an odd $\ell$ vanish by symmetry. We conclude that the claim holds for every even  $k$ satisfying $k> n=\frac{d-4}{2}$. In particular, $A^{(d)}_{k}\le 0$ for all even $k\ge 2$ provided $d\in \{4, 6\}$.
\end{proof}

\begin{remark}\label{rem:Raabe} In the proof of Lemma \ref{lem_FoschiEvNonPos}, we used the binomial expansion 
\begin{equation}\label{eq:binomial_expansion}
    (1-t^2)^\frac12=\sum_{m=0}^\infty \binom{\frac12}{m}(-1)^mt^{2m}.
\end{equation}
This series converges uniformly in  $[-1, 1]$. Indeed, it suffices to check that $\sum_{m\geq 0} |a_m| <\infty$, where $a_m:=\left\lvert \binom{1/2}{m}\right\rvert$.
In turn, this holds by the Raabe criterion:
\begin{equation}\label{eq:check_raabe}
    \lim_{m\to \infty} m\left( \frac{a_m}{a_{m+1}}-1\right)=\lim_{m\to \infty} m\frac{3/2}{m-1/2}=\frac32 >1, 
\end{equation}
and so $\sum_{m\geq 0} a_m <\infty$.
\end{remark}

We now prove that the eigenvalues of the kernels $K,L$ from \eqref{eq_kernels} have the claimed signs.
\begin{proof}[Proof of Lemma \ref{lem_eigenvalues}]
The case of $L$ is easier, so we focus on $K$.
Let $t=\omega_1\cdot \omega_2$.
    It will be enough to show that \begin{equation}\label{eq_magic_kernel_lemma}
        K_{2m}(t):=\int_{(\mathbb S^{d-1})^2} \left\lvert \sum_{j=1}^4 \omega_j\right\rvert^{2m}M(\boldsymbol{\omega})\, \d\sigma(\boldsymbol{\omega}_{34})=C_{d, m}t^{m+1}+P_{m}(t), 
    \end{equation}
    for some positive constant $C_{d,m}>0$ and polynomial $P_{m}$ of degree at most $m$. Indeed, recalling  \eqref{eq_lambda_k_definition}, there exists $C=C(d,k)>0$, such that 
    \begin{equation}\label{eq:recall_lambda_k_definition}
        \lambda_{2m}(k)=C\int_{-1}^1K_{2m}(t)C_k^{(\frac d 2-1)}(t)(1-t^2)^{\frac{d-3}{2}}\, \d t, 
    \end{equation} so~\eqref{eq_Rodrigues} and~\eqref{eq_magic_kernel_lemma}   together imply that $\lambda_{2m}(k)=0$ for $k>m+1$ and $\lambda_{2m}(m+1)>0$.
    
    Recall the definition~\eqref{eq_bigM}: $M(\boldsymbol{\omega}):=\frac14\left( \lvert \omega_1+\omega_2\rvert^2+\lvert \omega_3+\omega_4\rvert^2 -(\omega_1+\omega_2)\cdot(\omega_3+\omega_4)\right)$. We apply the trinomial expansion 
    \begin{equation}\label{eq:trinomial_exp} 
        (\alpha+\beta+\gamma)^m=\sum_{i+j+k=m} \binom{m}{i\,j\,k}\alpha^i\beta^j\gamma^k ,
    \end{equation}
    where the sum is taken over all nonnegative integers $i,j,k\geq 0$ satisfying $i+j+k=m$, and $\binom{m}{i\,j
    \,k}:=\frac{m!}{i!j!k!}$. Letting $\alpha=\lvert \omega_1+\omega_2\rvert^2, \beta=\lvert \omega_3+\omega_4\rvert^2, \gamma=2(\omega_1+\omega_2)\cdot(\omega_3+\omega_4)$, 
    \begin{align}
    K_{2m}(t)=&\sum_{i+j+k=m}\binom{m}{i\,j
    \,k}2^{k-2}|\omega_1+\omega_2|^{2i+2} \int_{(\mathbb S^{d-1})^2} |\omega_3+\omega_4|^{2j}((\omega_1+\omega_2)\cdot(\omega_3+\omega_4))^k\,\d\sigma(\boldsymbol{\omega}_{34})\notag\\
    &+\sum_{i+j+k=m}\binom{m}{i\,j
    \,k}2^{k-2}|\omega_1+\omega_2|^{2i} \int_{(\mathbb S^{d-1})^2} |\omega_3+\omega_4|^{2j+2}((\omega_1+\omega_2)\cdot(\omega_3+\omega_4))^k\,\d\sigma(\boldsymbol{\omega}_{34})\label{eq_interestingone}\\
    &-\sum_{i+j+k=m}\binom{m}{i\,j
    \,k}2^{k-2}|\omega_1+\omega_2|^{2i} \int_{(\mathbb S^{d-1})^2} |\omega_3+\omega_4|^{2j}((\omega_1+\omega_2)\cdot(\omega_3+\omega_4))^{k+1}\,\d\sigma(\boldsymbol{\omega}_{34}).  \notag
    \end{align}
     It follows that the highest degree term in $t=\frac12\lvert\omega_1+\omega_2\rvert^2 -1$ is obtained by setting $(i,j,k)=(m,0,0)$ on the first sum in \eqref{eq_interestingone}, yielding a contribution of 
    \[\frac14|\mathbb S^{d-1}|^2|\omega_1+\omega_2|^{2m+2}=\frac14|\mathbb S^{d-1}|^2(2+2t)^{m+1}=2^{m-1}|\mathbb S^{d-1}|^2 t^{m+1}+\textup{l.o.t.}\]
    Indeed, the remaining contributions  amount to a polynomial in $t$ of degree at most $m$ since, for integers $J,K\geq 0,$ the integrals in \eqref{eq_interestingone} are all of the form
    \begin{align*}
        &\int_{(\mathbb S^{d-1})^2} |\omega_3+\omega_4|^{2J}((\omega_1+\omega_2)\cdot(\omega_3+\omega_4))^{K}\,\d\sigma(\boldsymbol{\omega}_{34})\\
        &=\int_{\R^d}\int_{(\mathbb S^{d-1})^2}\boldsymbol{\delta}(x-\omega_3-\omega_4)|x|^{2J}((\omega_1+\omega_2)\cdot x)^K \d\sigma(\boldsymbol{\omega}_{34})\d x\\
        &=    \int_{|x|\leq 2} (\sigma\ast\sigma)(x) |x|^{2J} ((\omega_1+\omega_2)\cdot x)^K\,\d x =
            \begin{cases} C|\omega_1+\omega_2|^K, & K\text{ even}, \\
            0, &K\text{ odd},
            \end{cases}
    \end{align*}
    for some $C=C(d,J,K)>0$. To prove the last identity, we argue as follows. Since $\sigma\ast \sigma$ is a radial function,\footnote{More precisely: $(\sigma\ast\sigma)(x)=c_d|x|^{-1}(4-|x|^2)_+^{(d-3)/2}$, for a certain $c_d>0$; see \cite[Lemma 5]{COS15}. We implicitly used this to compute~\eqref{eq_explicit_K_one}; see \cite[Eq.\@ (4.2)]{CNOS21}} the last integral is of the form 
    \begin{equation}\label{eq:radial_integral}
        \Phi(\eta)=\int_{\lvert x \rvert\le 2} F(\lvert x \rvert)(\eta\cdot x)^K\, \d x,
    \end{equation}
    with $\eta=\omega_1+\omega_2$. If $K$ is odd, then this clearly vanishes. Otherwise, $\Phi$ must be radial and homogeneous of degree $K$, hence $\Phi(\eta)=C\lvert\eta\rvert^K$.  
    This completes the proof of \eqref{eq_magic_kernel_lemma} and therefore of Lemma \ref{lem_eigenvalues}.
\end{proof}

\section*{Acknowledgements}
This work was supported by the Research Institute for Mathematical Sciences, an
International Joint Usage/Research Center located in Kyoto University. 
GN and DOS are partially supported by FCT/Portugal through CAMGSD, IST-ID, projects UIDB/04459/2020 and UIDP/04459/2020. GN is also supported by FCT under the Scientific Employment Stimulus, Individual Call 2023.06691.CEECIND. DOS acknowledges further support from the IST Santander Start Up Funds, the Deutsche Forschungsgemeinschaft 
 (DFG, German Research Foundation) under Germany's Excellence Strategy – EXC-2047/1 – 390685813, 
and expresses his gratitude to Neal Bez and Yutaka Terasawa for organizing the RIMS Symposium on {\it Harmonic Analysis and Nonlinear Partial Differential Equations} (June 2024). The authors are grateful to the anonymous referee for valuable suggestions.

\end{document}